\documentclass[10pt]{amsart}
\usepackage{t1enc}
\usepackage[latin2]{inputenc}
\usepackage{amsmath}
\usepackage{amsthm}
\usepackage{amssymb}
\usepackage[T1]{fontenc}
\usepackage{enumerate}
\usepackage{verbatim}
\usepackage{amscd}
\usepackage{color}

\usepackage[colorlinks]{hyperref}
\usepackage[displaymath, tightpage]{preview}
\usepackage[colorinlistoftodos, textwidth=4cm, shadow]{todonotes}

\newcommand{\scr}{\mathbb}
\newcommand{\bld}{\textbf}

\newcommand{\cP}{\mathcal P}
\newcommand{\gM}{\mathfrak{M}}
\newcommand{\cF}{\mathcal F}

\newcommand{\con}{\mathfrak c}

\newcommand{\om}{\omega}

\newcommand{\ZFC}{\mathsf{ZFC}}

\newcommand{\AAA}{\protect{\mathcal A}}

\newcommand{\FF}{{\mathcal F}}

\newcommand{\MM}{{\mathcal M}}
\newcommand{\DD}{{\mathcal D}}

\newcommand{\ZZ}{{\mathcal Z}}

\newcommand{\er}{\mathbb R}

\newcommand{\supp}{{\rm supp} }

\newcommand{\conv}{{\rm conv}\,}

\newcommand{\sub}{\subseteq}

\newtheorem{thm}{Theorem}[section]

\newtheorem{lem}[thm]{Lemma}
\newtheorem{cor}[thm]{Corollary}
\newtheorem{prop}[thm]{Proposition}

\newtheorem{prob}[thm]{Problem}

\newtheorem*{clm}{Claim}
\newtheorem{claim}{Claim}

\begin{document}

\title{Sequential closure in the space of measures}

\thanks{Date: 19/03/2012}
\thanks{The first author was 
partially supported by the grant N N201 418939 from the Polish Ministry of Science and Higher Education.}

\author[Piotr Borodulin-Nadzieja]{Piotr Borodulin-Nadzieja}
\address{Instytut Matematyczny, Uniwersytet Wroc\l awski}
\email{pborod@math.uni.wroc.pl}

\author[Omar Selim]{Omar Selim}
\address{School of Mathematics, University of East Anglia}
\email{oselim.mth@gmail.com}

\begin{abstract} 
We show that there is a compact topological space carrying a measure which is not a weak$^*$ limit of finitely supported measures
but is in the sequential closure of the set of such measures. We construct compact spaces with measures of arbitrarily 
high levels of complexity in this sequential hierarchy. It follows that there is a compact space in which the sequential closure cannot be obtained in countably many steps. However, we show that this is not the case for our spaces where the sequential closure is always obtained in countably many steps.
\end{abstract}
\maketitle
\section{Introduction}\label{introduction}
\noindent
Given a compact Hausdorff topological space $K$ we denote by $\MM(K)$ the collection of all signed Borel measures on $K$. We let $C(K)$ be the collection of all real valued (bounded) continuous functions on $K$ equipped with the supremum norm.
By the Riesz representation theorem one can view $\MM(K)$ as the continuous dual of $C(K)$ and equip $\MM(K)$ with the weak$^*$ topology. Since $K$ forms a closed subspace of $\MM(K)$, the space $(\MM(K),\mbox{weak$^*$})$ is a natural topological vector space containing $K$.\\\\
We are interested in the sequential properties of $\MM(K)$. A sequence $(\mu_n)_{n}$ of measures converges to $\mu$ in the weak$^*$ topology precisely when for every $f\in C(K)$ one has
$$
\int_K f d \mu_n \rightarrow \int_K f d \mu.
$$
If in addition $K$ is zero-dimensional, for example if $K$ is the Stone space of a Boolean algebra, then the above condition is equivalent to asserting that $\mu_n(C)\rightarrow \mu(C)$ for every clopen set $C$.\\\\
Rather than $\MM(K)$ we will instead consider the collection $P(K)\subseteq \MM(K)$ of (non-negative) probability measures on $K$. The space $P(K)$ is a slightly simpler object to deal with. This is justified, since $P(K)$ is a closed subspace of $\MM(K)$ and for every $\mu\in \MM(K)$ one can find unique $\mu_1,\mu_2\in P(K)$ such that $\mu = c_1\mu_1 - c_2\mu_2$, for appropriate constants $c_1,c_2\in \scr{R}$. For details on the above discussion see \cite{Bogachev}.\footnote{However, notice that in \cite{Bogachev} and in many other sources ``weak$^*$ convergence'' is called ``weak 
convergence''}\\\\
Spaces of measures are in general very complicated. Indeed, even for well known spaces $K$ there are many simply stated problems about $P(K)$ which remain open (see for example \cite{Koszmider-survey}). Arguably the simplest measures in $P(K)$ are the so called \emph{Dirac} measures of the form 
$$
\delta_x(A) = \left\{\begin{array}{cl} 1 ,\quad&\mbox{if
$x\in A$;}\\
0,\quad&\mbox{otherwise.}\end{array}\right.
$$
for $x\in K$. The collection of Dirac measures forms a closed copy of $K$ in $P(K)$. A slightly more complicated space is the collection of measures of finite support. These are linear combinations of Dirac measures with positive coefficients that sum to $1$. If $\MM(K)$ is given the usual vector lattice structure then the measures of finite
support may be viewed as the convex hull of the collection of Dirac measures. Accordingly for $A\subseteq K$ we let $\conv(\delta_x\colon x\in A)$ be the collection of probability measures finitely supported on $A$.\\\\
We are interested in the following sequential hierarchy. For $A\subseteq K$ let  $S_0(A) = \conv(\delta_x\colon x\in A)$. For $\alpha\leq \omega_1$ if $S_\beta$ has been defined for each $\beta<\alpha$ then let 
$$
S_\alpha(A) = \{\lim_n \mu_n\colon (\forall n)(\exists \beta<\alpha)(\mu_n\in S_\beta(A))\}.
$$
We let $S(A) = S_{\omega_1}(A)$.\\\\
The collection $S_0(K)$ will be dense in $K$ but since $P(K)$ is not necessarily a sequential space in general we might not have $S_1(K) = P(K)$. To illustrate we have the following simple fact. 
\begin{prop} \label{carrier} 
Assume $\mu$ is a probability measure on $K$. 
If $\mu\in S(K)$, then $\mu$ has a separable closed subset of full measure.
\end{prop}
\noindent
In particular, it follows that if $\mu$ is a strictly positive probability measure on a non-separable space $K$, 
then $\mu \in P(K)\setminus S(K)$. For example, if $\gM$ is the complete Boolean algebra of Borel subsets of the unit interval modulo Lebesgue null sets, $\lambda$ is the Lebesgue measure and $K$ is the Stone space of $\gM$ then $\lambda\in P(K)\setminus S(K)$.\\\\
Actually we can prove a more general fact here. Recall that a measure $\mu$ on a space $K$ is purely atomic if it is of the form 
\[ \mu = \sum_{n\in\om} a_n \delta_{x_n} \]
for $\{x_n\colon n\in\om\}\sub K$ and a sequence of positive real numbers $(a_n)_{n\in\om}$. Recall also that a space is extremely disconnected if closures of open sets remain open. 
\begin{prop} \label{extremely_disconnected}
If $K$ is a compact extremely disconnected Hausdorff space, then every measure in $S(K)$ is
purely atomic.
\end{prop}
\begin{proof}
Let $\mu\in S(K)$. If $K$ is compact and extremely disconnected, then $C(K)$ is Grothendieck (that is to say, if a sequence $(\mu_n)_{n\in\omega}\subseteq P(K)$ weak$^*$ converges to $\mu$ then for every bounded Borel function $f\colon K\rightarrow \scr{R}$ we have $\int_K f d {\mu_n} \rightarrow \int_K f d\mu$, see \cite[Theorem 9]{grothendieck}).
Hence, if $\mu\in P(K)$ is a limit of a sequence $(\mu_n)_{n\in\omega}$ of finitely supported measures and
$F = \bigcup_{n\in\omega} \supp(\mu_n)$, then
\[\int_K \chi_F \ d\mu = 1. \]
Therefore, 
\[ \mu = \sum_{n\in\om} a_n \delta_{x_n} ,\] 
where $F=\{x_n\colon n\in\om\}$. Similarly, one can prove that all measures in $S(K)$ are purely atomic.
\end{proof}
\noindent
Since the Stone space of a complete Boolean algebra is extremely disconnected (\cite[Proposition 7.21]{handbook}) we can conclude
that every non-atomic measure on the Stone space of $\gM$ or indeed $\beta\omega$ is not in the sequential closure of the finitely supported measures. Recall that a measure on a Boolean algebra is non-atomic if one can find finite partitions of unity consisting of sets of arbitrarily small measure.\\\\
On the other hand we have the following examples:
\begin{itemize}
\item if $K$ is scattered or metrisable then $P(K) = S_1(K)$ (see for example \cite{Mercourakis});
\item $P(2^{\om_1}) = S_1(2^{\om_1})$ (see \cite[Theorem 2]{Losert});
\item $P(2^\con) = S_1(2^\con)$ (see \cite[491Q]{Fremlin-MT4});
\item if $K$ is Koppelberg compact (that is to say, it is a Stone space of a minimally generated Boolean algebra) then $P(K) = S(K)$ (see \cite[Theorem 5.1]{Pbn1}).
\end{itemize}
It is natural to ask to what extent can one control the equalities between the $S_\alpha(K)$ for different $\alpha$, by choosing $K$ appropriately. Recently Avil\'{e}s, Plebanek and Rodr\'{\i}guez proved the following (see \cite{APR}).
\begin{thm} \label{grzes} 
Assuming the continuum hypothesis there exists a space $K$
such that \[ S_1(K)\subsetneq S_2(K) \sub S_3(K) = P(K).\]
\end{thm}
\noindent
The aim of this article is to prove the following.
\begin{thm}\label{mainfromone} For each $1\leq \alpha<\omega_1$ there exists a Stone space (compact Hausdorff zero-dimensional) space $K^{(\alpha)}$ and a measure $\mu^{(\alpha)}$ such that
	\begin{itemize}
 \item[(1)] $\mu^{(\alpha)}\in S_\alpha(K^{(\alpha)})\setminus \bigcup_{\beta<\alpha} S_\beta(K^{(\beta)})$
 \item[(2)] $S(K^{(\alpha)}) = S_{\alpha+1}(K^{(\alpha)})$.
\end{itemize}
\end{thm}
\noindent
Unfortunately, the above constructed spaces will contain clopen copies of $\beta\omega$ and so by Proposition \ref{extremely_disconnected} one will have $S(K^{(\alpha)})\neq P(K^{(\alpha)})$.\\\\
\noindent
In Section 2 we construct the sequence $(K^{(\alpha)},\mu^{(\alpha)})_{1\leq \alpha<\omega_1}$ and show that property (1) holds. In Section 3 we show that for these space property (2) holds.
\section{Non-triviality of the sequential hierarchy} \label{hierarchy}
\noindent
In this section we will prove that the sequential hierarchy of measures is non-trivial in $\ZFC$
and that measures of arbitrary sequential complexity can be constructed. We shall construct the sequence $(K^{(\alpha)},\mu^{(\alpha)})_{\alpha}$ from Theorem \ref{mainfromone} and show that for each $\alpha$ we have 
\begin{equation}\label{m1}
\mu^{(\alpha)}\in S_\alpha(K^{(\alpha)})\setminus \bigcup_{\beta<\alpha}S_\beta(K^{(\alpha)}).
\end{equation}
\noindent
We will need following observation several times for the remainder of this article, the proof of which is straightforward.
\begin{lem}\label{3} 
Let $K$ be a zero-dimensional space and let $Y$ be a non-empty clopen subset of $K$. 
If $\mu \in S_\alpha(Y)$ then $\mu'$ defined by $\mu'(A) = \mu(Y\cap A)$ is a member 
of $S_\alpha(K)$. Conversely, if $\mu\in S_\alpha(K)$ and $\mu(Y)\neq 0$ then 
$$
\frac{1}{\mu(Y)}\mu\restriction Y \in S_\alpha(Y).
$$
\end{lem}
\noindent
As a corollary to (\ref{m1}) we have the following.
\begin{cor} \label{omega1} There exists a Stone space $K^{(\omega_1)}$ such that for every $1\leq \alpha<\omega_1$ we have $S_\alpha(K)\setminus \bigcup_{\beta<\alpha}S_\beta(K)\neq \emptyset$. 
\end{cor}
\begin{proof} Let $(K^{(\alpha)},\mu^{(\alpha)})_{\alpha}$ be the sequence promised by Theorem \ref{mainfromone} and let $K$ be the one-point
compactification of the disjoint union $\bigsqcup_{\alpha<\omega_1} K^{(\alpha)}$. Define $\mu'_\alpha\colon \mathrm{Clopen}(K)\rightarrow [0,1]$ by $\mu'_\alpha(A) = \mu^{(\alpha)}(A\cap K^{(\alpha)})$, then by Lemma \ref{3} for every $\alpha<\omega_1$
$$
\mu'_\alpha\in S_{\alpha}(K)\setminus \bigcup_{\beta<\alpha}S_{\alpha}(K).
$$
\end{proof}
\noindent
Let us now fix some terminology and notation. If $\scr{B}$ is a Boolean algebra then by $\mathrm{Stone}(\scr{B})$ we will mean the usual Stone space of $K$ consisting of all ultrafilters on $\scr{B}$. If $K$ is any topological space then by $\mathrm{Clopen}(K)$ and $\mathrm{Borel}(K)$ we denote the collection of clopen subsets of $K$ and the collection of Borel subsets of $K$, respectively.\\\\
We will consider subalgebras $\scr{B}$ of $\cP(\om)$ and finitely additive
probability measures on $\scr{B}$. Recall that a finitely additive measure defined on a Boolean algebra $\scr{B}$ defines a finitely additive measure on the clopen sets of $K = \mathrm{Stone}(\scr{B})$ and that this can be uniquely extended to a $\sigma$-additive Radon measure on the Baire algebra of $K$, the $\sigma$-algebra generated by $\mathrm{Clopen}(K)$. Since $K$ is compact and Hausdorff this measure uniquely extends to a $\sigma$-additive Radon measure on $\mathrm{Borel}(K)$. Once again, for details see \cite{Bogachev}.\\\\ 
Unless otherwise stated, all sequences are assumed to be indexed by $\omega$. For simplicity we will usually identify a (finitely additive) measure on $\scr{B}$ with the corresponding ($\sigma$-additive) measure on $\mathrm{Borel}(\!\mathrm{Stone}(\scr{B}))$. We will also identify elements of a Boolean algebra with the corresponding
clopen subsets of its Stone space.  Given a subalgebra $\scr{B}$ of $\cP(\omega)$, so that there is no confusion between unions in $\cP(\omega)$ and unions in $\mathrm{Stone}(\scr{B})$, if $(B_n)_{n}$ is a sequence in $\scr{B}$ then
by  $\widehat{\bigcup}_{n} B_n$ we will mean the union in $\mathrm{Stone}(\scr{B})$ of the corresponding sequence of clopen sets.  Similarly we define $\widehat{\bigcap}_n B_n$. When considering a subalgebra $\scr{B}$ of $\cP(\omega)$ such that $[\omega]^{<\omega} \subseteq \scr{B}$ we will
identify each natural number $n$ with the corresponding principle ultrafilter $\{B\in \scr{B}\colon  n\in B\}$.\\\\
Recall that a measure $\nu$ on a Boolean algebra $\scr{B}$ is \emph{orthogonal} to a measure $\mu$ on $\scr{B}$ if
for each $\delta>0$ there exists an $F\in \scr{B}$ such that $\nu(F)>1-\delta$ and $\mu(F)<\delta$. Let $\scr{A}$ and $\scr{B}$ be Boolean algebras carrying measures $\nu$ and $\mu$, respectively. We will say that $(\scr{A},\nu)$ is \emph{metrically isomorphic} to $(\scr{B},\mu)$ if there exists a map $\varphi\colon
\scr{A} \to \scr{B}$ such that the map $A\mapsto [\varphi(A)]_{\mathrm{Null}(\mu)}$ is a measure preserving isomorphism (i.e. such that $\mu([\varphi(A)]_{\mathrm{Null}(\mu)})=\nu(A)$ for each $A\in \scr{A}$).\\\\
We now describe the general construction by which we shall obtain the sequence $(K^{(\alpha)},\mu^{(\alpha)})_{\alpha}$ of Theorem \ref{mainfromone}. Fix once and for all a sequence $(B_n)_n$ consisting of infinite and pairwise disjoint subsets of
$\omega$ such that $\bigcup_{n}B_n = \omega$. For each $n$ let $B_n = \{x^n_0<x^n_1<\cdots\}$. For $A\in \cP(\omega)$ and $n\in\omega$ let $A^n = \{x^n_i\colon i\in A\}$ and for $\AAA\subseteq \cP(\omega)$ let $\AAA^n = \{A^n\colon A\in\AAA\}$. Say for each $n$
we have a subalgebra $\scr{B}_n$ of $\cP(\omega)$ carrying a measure $\mu_n$ such that $(\gM,\lambda)$ is metrically isomorphic to each $(\scr{B}_n,\mu_n)$, witnessed by $\varphi_n\colon \gM \to \scr{B}_n$. Recall from Section \ref{introduction} that
$\gM$ is the Boolean algebra of Borel subsets of the unit interval modulo Lebesgue null sets and that $\lambda$ is the Lebesgue measure on $\gM$. For each $n$ let $\mu'_n\colon \scr{B}_n^n\rightarrow \scr{R}$ be the measure defined by 
$$
\mu'_n(A) = \mu_n(\{i\colon x^n_i\in A\}).
$$
Let 
$$
\cF = \{A\subseteq \omega\colon A\cap B_n\in\scr{B}_n^n\mbox{ and } \lim_{n\rightarrow \infty}\mu'_n(A\cap B_n) = 1\}
$$
and 
$$
\cF' = \{\bigcup_{n\in\omega} \varphi_{n}(M)^n\colon M\in \gM\}.
$$
Finally, let $\scr{B}$ be the subalgebra of $\cP(\omega)$ generated by $\cF\cup \cF'$ and let $\mu\colon \scr{B}\rightarrow \scr{R}$ be the measure 
$$
\mu(A)= \lim_{n\rightarrow \infty}\mu'_n(A\cap B_n).
$$
We will call $(\scr{B},\mu)$ the pair constructed \emph{canonically} from $(\scr{B}_n,\mu_n)_n$. The crucial information about this construction is summarised by the following theorem. 
\begin{thm} \label{Main}
For each $n$ let $\scr{B}_n$ be a subalgebra of $\cP(\omega)$ carrying a measure $\mu_n$ and let $K_n = \mathrm{Stone}(\scr{B}_n)$. Let $(\alpha_n)_{n}$ be a non-decreasing sequence of non-zero countable ordinals with $\alpha = \sup_{n} \alpha_n$. Assume that for each $n$
\begin{itemize}
 \item $\mu_n \in S_{\alpha_n}(\omega)$;
 \item $\mu_n$ is orthogonal to each member of $\bigcup_{\beta<\alpha_n} S_{\beta}(K_n)$;
 \item $(\gM, \lambda)$ is metrically isomorphic to $(\scr{B}_n,\mu_n)$ witnessed by $\varphi_n$.
\end{itemize}
Let $(\scr{B},\mu)$ be the pair constructed canonically from $(\scr{B}_n,\mu_n)_n$ and let $K = \mathrm{Stone}(\scr{B})$. Then the following hold:
\begin{itemize}
 \item $\mu\in \left\{\begin{array}{cl} S_{\alpha+1}(\omega),\quad&\mbox{if $(\exists n)(\alpha = \alpha_n)$, }\\
S_\alpha(\omega),\quad&\mbox{otherwise.}\end{array}\right.$
 \item $\mu$ is orthogonal to each member of $\left\{\begin{array}{cl} \bigcup_{\beta\leq \alpha} S_{\alpha}(K),\quad&\mbox{if $(\exists n)(\alpha = \alpha_n)$,}\\
\bigcup_{\beta<\alpha} S_{\alpha}(K),\quad&\mbox{otherwise.}\end{array}\right.$
 \item $(\gM, \lambda)$ is metrically isomorphic to $(\scr{B},\mu)$.
\end{itemize}
\end{thm}
\noindent
Assuming Theorem \ref{Main} for now we can proceed to the construction of $(K^{(\alpha)},\mu^{(\alpha)})_{\alpha}$. Recall that for $A\in\cP(\omega)$ the asymptotic density function $d$ is given
by
\[ d(A)= \lim_{n\to\infty} |A\cap \{0,1,\dots,n\}|/(n+1), \]
should this limit exist. Let $\DD$ be the family of sets with 
asymptotic density and
let $\ZZ$ be the ideal of asymptotic density $0$ sets. Notice that $\DD$ is not an algebra. Let $\varphi\colon \gM \rightarrow \DD$ be such that the map
$$
M\mapsto [\varphi(M)]_{\ZZ}\colon \gM\rightarrow \cP(\omega)/\ZZ
$$
is a monomorphism satisfying $\lambda(M) = d(\varphi(M))$, for every $M\in \gM$ (see \cite[Theorem 2.1]{Frank} or \cite[491N]{Fremlin-MT4}).  Let $\scr{A}$ be the Boolean algebra generated by
$\ZZ$ and the set $\{\varphi(M)\colon M\in
\mathfrak{M}\}$. Let $K^{(1)} = \mathrm{Stone}(\scr{A})$ and $\mu^{(1)}$ be the measure on
$\scr{A}$ defined by
\[ \mu^{(1)}(A) = d(A). \]
It is easy to check that $\mu^{(1)} \in S_1(\om)$. Indeed for every
$A\in \scr{A}$ we have
\[ \mu^{(1)}(A) = \lim_{n\to\infty}
\frac{1}{n+1}(\delta_{0}(A)+\ldots
+ \delta_{n}(A)). \]
Plainly, $\mu^{(1)}$ is orthogonal to every finitely
supported measure because each ultrafilter in $K^{(1)}$ contains a set of arbitrarily small density. If for each $n$ we take $\scr{B}_n = \scr{A}$, $\mu_n =
\nu$ and $\alpha_n = 1$ then the conditions of Theorem \ref{Main} are satisfied and we can
start the induction by letting $(\scr{B},\mu)$ be the pair canonically constructed from $(\scr{B}_n,\mu_n)_{n}$ and by setting $\mu^{(2)} = \mu$ and $K^{(2)} = \mathrm{Stone}(\scr{B})$. At limit stage $\gamma$ we
let $\{\alpha_n\colon n\in \om\}$ be an increasing sequence cofinal in $\gamma$
and again use Theorem \ref{Main}. This proves (\ref{m1}) and therefore the first part of Theorem \ref{mainfromone}.\\\\
Before we move on to the proof of Theorem \ref{Main} let us fix one more piece of notation. If $(\scr{B},\mu)$ has been constructed canonically from $(\scr{B}_n,\mu_n)_n$ then we let $L_0$ be the collection of all non-negative $\nu \in \MM(\mathrm{Stone}(\scr{B}))$ such that $\widehat{\bigcup}_n B_n$ is $\nu$-full. We let $L_1$ be the collection of all non-negative $\nu \in \MM(\mathrm{Stone}(\scr{B}))$ such that $\nu(\widehat{\bigcup}_n B_n) = 0$. Clearly for any non-negative $\nu\in \MM(K)$ there exists $\nu_0\in L_0$ and $\nu_1\in L_1$ such that $\nu = \nu_0+\nu_1$, indeed, just take $\nu_0 = \nu\restriction \widehat{\bigcup}_n B_n$ and $\nu_1 = \nu-\nu_0$. This notation of course depends on the algebra $\scr{B}$, but we hope that the choice of $\scr{B}$ will be clear from the context.
\noindent
\begin{proof}[Proof of Theorem \ref{Main}] For clarity's sake first notice that if $\lambda\in S_\alpha(K)$ for some $\alpha$ and $\lambda(B_n)\neq 0$ then by Lemma \ref{3} 
$$
\frac{1}{\lambda(B_n)}\lambda \restriction B_n \in S_\alpha(B_n)
$$
and thus $\lambda\restriction B_n$ can be viewed as a member of $S_\alpha(K_n)$. If $\lambda \in S_\alpha(K_n)$ then in the same way $\lambda$ can be viewed as a measure in $S_\alpha(B_n)$. By Lemma \ref{3} the measure $\lambda'$ defined by $\lambda'(A) = \lambda(A\cap B_n)$ is a measure in $S_\alpha(K)$.\\\\
Let $R \sub K$ be the collection of ultrafilters $u\in K$ such that $\cF\subseteq u$. Then $R$ is a closed set in $K$ and is indeed the support of $\mu$.
\begin{claim}\label{D}
For every countable subset $S$ of $R$ and $\delta>0$ there is $F_0\in \scr{B}$ such that
$F_0\cap S=\emptyset$ but $\mu(F_0)>1-\delta$.
\end{claim}
\begin{proof} 
First notice that since $\mu$ is Radon
\[ \mathrm{Clopen}(R) \cong \scr{B}/\mathrm{Null}(\mu) = \{[\bigcup_n \varphi_n(M)^n]_{\mathrm{Null}(\mu)}\colon M\in \gM\}. \]	
The latter is clearly isomorphic to $\gM$ and
therefore $\mathfrak{M}$ is isomorphic to $\mathrm{Clopen}(R)$ and it is easy to see that this isomorphism transfers $\mu$ to
$\lambda$. Thus $R$ is homeomorphic to $\mathrm{Stone}(\mathfrak{M})$ and the measure $\mu$ on $R$ can be seen as the Lebesgue measure on $\mathrm{Stone}(\mathfrak{M})$.\\\\
Let $S$ be countable and $\delta>0$. Of course $\mu(S)=0$. Since the Lebesgue measure is clopen regular on $\mathrm{Borel}(R)$ (see \cite[322Qc]{fremv3}) we can find a set $C$ clopen in $R$ which is disjoint from $S$
and is such that $\mu(C)> 1-\delta$. Let $M\in \gM$ such that $C = [\bigcup_n \varphi_n(M)^n]_{\mathrm{Null}(\mu)}$ and $F_0 = \bigcup_n \varphi_n(M)^n\in \scr{B}$. Since $F_0\cap R = C$, $F_0$ is disjoint from $S$ and clearly $\mu(F_0)>1-\delta$.
\end{proof}
\noindent
For measures $\nu_1$ and $\nu_2$ defined on a Boolean algebra $\scr{C}$ we say that $\nu_1$ is \emph{strongly orthogonal} to $\nu_2$ if for each $\delta>0$ there is $C\in\scr{C}$ such that $\nu_1(C)>\nu_1(\scr{C})-\delta$ and $\nu_2(C)=0$.\\\\
In case of measures from $S(K\setminus R)$, the crux of the proof lies in the following claim.
\begin{claim}\label{claim2} If $(\lambda_k)_k$ is a sequence from $P(K)$ such that for all but finitely many $n$, $\mu'_n$ is orthogonal to each $\lambda_k\restriction B_n$, and each $\lambda_k$ is strongly orthogonal to $\mu$, then for every $\delta>0$ we can find an
$F\in \FF$ and a subsequence $(\lambda_{k_n})_{n}$ such that $\lambda_{k_n}(F)\leq \delta$ for each $n$.
\end{claim}
\noindent
We divide the proof of Claim \ref{claim2} into two parts considering measures from $L_1$ and measures from $L_0$ separately.
\begin{claim}\label{claim3} If $(\lambda_k)_{k}$ is a sequence from $L_1$ and 
each $\lambda_k$ is strongly orthogonal to $\mu$, then for each $\delta>0$
there exists $F_1\in \FF$ such that for each $k$ we have $\lambda_k(F_1)<\delta$.
\end{claim}
\begin{proof} Fix $\delta>0$. Since $(\lambda_n)_n$ is strongly orthogonal to $\mu$ we can find, for each $k$, a set $V_k\in \scr{B}$ such that
$\lambda_k(V_k)>\lambda_k(K)-\delta$ and $\lim_n \mu'_n(V_k\cap B_n) = 0$. 
Clearly, $\lim_n \mu'_n(\bigcup_{k\leq l} V_k \cap
B_n) = 0$ for each $l\in \om$. Therefore, there is
a sequence $(n_i)_i$ of natural numbers such
that
\begin{itemize}
	\item $n_i<n_{i+1}$ for each $i$;
	\item for every $n\geq n_i$ we have
		$\mu_n'(\bigcup_{j\leq i} V_j \cap B_n) <
		1/(i+1)$.
\end{itemize}
For each $i$ and $n\in [n_i,n_{i+1})\cap \omega$, let $C_n = B_n\setminus (\bigcup_{j\leq i} V_j) \in\scr{B}_n^n$. Now set $F_1 = \bigcup_{n\geq n_0} C_n$. Then for each $i$ and $n\geq n_i$ we have $\mu_n'(F_1 \cap B_n)> 1-1/(i+1)$ so
that $F_1\in \cF$. Since for each $k$ we have $V_k\setminus \bigcup_{l<k} B_l\subseteq K\setminus F_1$ it follows that $\lambda_k(F_1)<\delta$.
This completes the proof of Claim \ref{claim3}.
\end{proof}
\begin{claim}\label{claim4} Suppose $(\lambda_k)_k$ is a subsequence from $L_0$ such that for all but finitely many $n$, $\mu'_n$ is orthogonal to each $\lambda_k\restriction B_n$. Then for each $\delta>0$ there is an $F_2\in \FF$ such that $\lambda_k(F_2)\leq\delta$ for infinitely many $k$.
\end{claim}
\begin{proof} Fix $\delta>0$. A straightforward computation shows that we can find $N\in\omega$ such that for each $j$ there are infinitely many $k$ where
$$
\lambda_k(\bigcup_{N\leq i\leq j}B_i)<\delta/2.
$$
In light of this and the orthogonality assumption on the $\lambda_k$ we may, without loss of generality, assume that for every $j$ there are
infinitely many $k$'s such that
$\lambda_k(\bigcup_{i\leq j} B_i) < \delta/2$ and also that for every $n$, the measure $\mu'_n$ is orthogonal to every $\lambda_k\restriction B_n$.\\\\
We construct a sequence $(k_n)_{n}$ of natural numbers and a sequence
$(A_n)_{n}$ such that
\begin{itemize}
	\item $k_n< k_{n+1}$ for each $n$;
	\item $A_n \in \scr{B}^n_n$ and $\mu'_n(A_n)>1 -
		1/(n+1)$ for each $n$;
	\item for every $n$ and $l > n$ we have $\lambda_{k_n}(\bigcup_{i\leq n}A_i)<\delta/2$ and $\lambda_{k_{n}}(A_{l})<\delta/2^{l+1}$.
\end{itemize}
Let $A_0 = B_0$ and let $k_0$ be such that
$\lambda_{k_0}(A_0)<\delta/2$. Suppose we have constructed
$k_n$ and $A_n$ as above. Since $\mu'_{n+1}$ is
orthogonal to $\lambda_{k_j}\restriction B_{n+1}$ 
for each $j\leq n$,
there is $A_{n+1}\in \scr{B}^{n+1}_{n+1}$ such that 
\begin{itemize}
	\item $\mu'_{n+1}(A_{n+1})>1 - 1/(n+2)$;
	\item $\lambda_{k_j}(A_{n+1})<\delta/2^{n+2}$ for
		each $j\leq n$.
\end{itemize}
Let $F_2=\bigcup_n A_n$. Clearly $F_2 \in \FF$ and for any $n$, since $\lambda_{k_n} \in L_0$, we have
\[ \lambda_{k_n}(F_2) = \lim_{m\to \infty} \lambda_{k_n}(\bigcup_{n<m} A_n) \leq \delta. \]
This completes the proof of Claim \ref{claim4}.
\end{proof}
\noindent
Claim \ref{claim2} now follows since each $\lambda_k$ can be decomposed into a member from $L_0$ and $L_1$. We now have the following two cases.\\\\
\noindent
\bld{Case 1.} $\alpha_n<\alpha$ for each $n$.\\\\
Since each $\mu'_n\in S_{\alpha_n}(\omega)$ (with respect to $K_n$, and therefore with respect to $K$), we have $\mu \in S_{\alpha}(\omega)$ (with respect to $K$). Let $\nu\in S_{\beta}(K)$ for some $\beta<\alpha$. Then for some $N\in\omega$ we have $\beta<\alpha_n$ for each $n>N$. 
Fix $\delta>0$ and let $S$ be the countable set of points in $R$ which appear
in the process of constructing $\nu$. Using Claim \ref{D} we can find a clopen $F_0$ such that
$\mu(F_0)>1-\delta$ and $F_0\cap S=\emptyset$. The measure $\frac{1}{\nu(F_0)} \nu\restriction F_0\in S_\beta(K\setminus R)$.\\\\
Notice that, by the definition of $R$, all measures in $S_0(K\setminus R)$ are strongly orthogonal to $\mu$. By a repeated application of Claim \ref{claim2} and the assumption that, for each $n>N$, measures from $\bigcup_{\beta<\alpha_n}S_\beta(B_n)$ are orthogonal to $\mu'_n$, it follows by induction that all measures in $S_\beta(K \setminus R)$ are strongly orthogonal to $\mu$. Thus we can find $F_1\in \scr{B}$ such 
that $(\nu\restriction F_0)(F_1)<\delta$ and $\mu(F_1)=1$. In particular $\nu(F_0\cap F_1)<\delta$ but $\mu(F_0\cap F_1)>1-\delta$ and $\mu$ is orthogonal to $\nu$.\\\\
\bld{Case 2.} $\alpha_n=\alpha$ for some $n$.\\\\
If $\nu\in S_\alpha(K)$, then it is a limit of measures from $\bigcup_{\beta<\alpha} S_\beta(K)$ and we can proceed as in Case 1.\\\\
Finally, as in Claim \ref{D}, we see that $(\gM,\lambda)$ is metrically isomorphic to $(\scr{B}, \mu)$ which is witnessed by the map
\[ M\mapsto \bigcup_n \varphi_n(M)^n. \]
This completes the proof of Theorem \ref{Main}. 
\end{proof}
\section{Upper bounds of the sequential hierarchy}
\noindent
In the previous section we constructed the sequence $(K^{(\alpha)},\mu^{(\alpha)})_{\alpha}$ of Theorem \ref{mainfromone} witnessing the non-triviality of the sequential hierarchy. A priori it could be that even in $P(K^{(2)})$ there
are measures of arbitrarily high sequential complexity (like in $P(K^{(\om_1)})$ from Corollary \ref{omega1}). Then the construction of $K^{(\alpha)}$ for $\alpha>2$ would be obsolete. This is not the case. In this section we prove the second part of Theorem \ref{mainfromone}, that is, we show that for each $1\leq\alpha<\omega_1$ we have 
\begin{equation}\label{m2}
S(K^{(\alpha)}) = S_{\alpha+1}(K^{(\alpha)}).
\end{equation}
\noindent
Recall that a non-negative Borel measure $\nu$ is \emph{absolutely continuous} with respect to a non-negative Borel measure $\mu$, and we write $\nu\ll \mu$, if $\mathrm{Null}(\mu)\subseteq \mathrm{Null}(\nu)$. We will need the following fact.
\begin{thm}[Plebanek] \label{plebanek}
Let $K$ be a Stone space and let $\mu\in S_\alpha(K)$ for some $\alpha<\om_1$.
If $\nu \ll \mu$ then 
$\nu\in S_{\alpha+1}(K)$.
\end{thm}
\begin{proof}
By the Radon-Nikodym theorem we can find a function $f\colon K\to \er$ such that
\[ \nu(A) = \int f \cdot \chi_A \ d\mu \]
for every clopen $A\sub K$.
\begin{clm} If $f$ is continuous then $\nu\in S_\alpha(K)$. 
\end{clm}
\begin{proof}
Assume that $\mu\in S_1(K)$ and
$(\mu_n)_n$ is a sequence of finitely supported measures converging to $\mu$. Define
\[ \nu_n(A) = \int f \cdot \chi_A \ d\mu_n \]
for each $n$ and clopen $A\sub K$. Then for every clopen $A\sub K$ the sequence $(\nu_n(A))_n$ converges to $\nu(A)$, 
since each $f\cdot \chi_A$ is continuous. Hence, $(\nu_n)_n$ converges to $\nu$ in the weak$^*$ topology.
It is also easy to check that $\nu_n$ is finitely supported for every $n$. Proceed by induction.
\end{proof}
\noindent
Consider now the general case, when $f\colon K\to \er$ is measurable (with respect to the completion of $\mu$).
Recall that since $K$ is compact and $\mu$ is Radon the space of continuous functions is dense in $L^p(K,\mu)$ whenever $0<p<\infty$ (see \cite[Corollary 4.2.2]{Bogachev}). Therefore, there is a sequence $(f_n)_n$ of real valued continuous functions on $K$ that converges to $f$ in $L^1(\mu)$. For every $n$ and clopen $A\sub K$ define 
\[ \nu_n(A) = \int f_n \cdot \chi_A d\mu. \]
From the claim it follows that the sequence $(\nu_n)_n$ consists of measures from $S_\alpha(K)$. It is also convergent to $\nu$ and so $\nu\in S_{\alpha+1}(K)$.
\end{proof}
\noindent
Theorem \ref{plebanek} is true not only for zero-dimensional spaces, but this assumption slightly shortens the proof and it is enough for our purposes. It is unclear to us if the conclusion of this theorem can be strengthen to $\nu\in S_\alpha(K)$.\\\\
Recall from Section \ref{introduction} that a measure $\nu$ on a Boolean algebra $\scr{A}$ is \emph{non-atomic} if for each $\varepsilon>0$ there exists a finite partition $(A_i)_{i<n}$ of $\scr{A}$ such that $\nu(A_i)<\varepsilon$ for each $i$. By
the Hammer-Sobczyk decomposition theorem every measure on a Boolean algebra is a (unique) sum of a purely atomic measure and a non-atomic one (see \cite{sobczyk} and \cite[Theorem 5.2.7]{rao}). With this in mind we shall henceforth refer to the
\emph{non-atomic (or atomic) part of a measure}. We will also need the following.
\begin{lem} \label{difference} Let $K$ be a Stone space and let $\nu\in S_\alpha(K)$, for some $\alpha$, be a measure that is not purely atomic. If $\nu''$ is the non-atomic part of $\nu$ then \[ \frac{1}{\nu''(K)} \nu'' \in S_\alpha(K).\]
\end{lem}
\begin{proof} Let $\nu'$ be the atomic part of $\nu$. For every $n$ using the non-atomicity of $\nu''$ we can find a clopen subset $A_n$ of $K$ such that $\nu''(A_n)<1/n$ and $\nu'(K\setminus A_n)<1/n$.\\\\
Assume that $(\nu_k)_k$ witnesses that $\nu\in S_\alpha(K)$. For each $k,n$ we have $\nu(A_n\setminus A_k) = \nu'(A_n \setminus A_k) + \nu''(A_n\setminus A_k) < 1/k + 1/n \leq 2/n$ and similarly $\nu(A_k\setminus A_n) < 1/n + 1/k \leq 2/n$. Thus, we can assume (considering a subsequence if needed) that 
$\nu_k(A_n \setminus A_k)<2/n$ and $\nu_k(A_k\setminus A_n)<2/n$ for each $n<k$.
For each $k$ define $\nu''_k = \nu_k |(K\setminus A_k)$. We will show that $(\nu''_k)_k$ converges to $\nu''$. 
Indeed, let $C$ be a clopen subset of $K$ and $\varepsilon>0$. Find $n$ such that $2/n<\varepsilon/5$. 
Denote  $C' = C\setminus A_n$.
Let $K>n$ be such that for every $k>K$ we have $|\nu_k(C')-\nu(C')|<\varepsilon/5$. We have 
\[ |\nu''(C) - \nu''(C')| \leq \nu''(A_n) < \varepsilon/5 \]
and
\[ |\nu''_k(C) - \nu''_k(C')|  =  |\nu_k(C\setminus A_k) - \nu_k(C\setminus (A_n \cup A_k))| 
\leq \nu_k(A_n\setminus A_k) < \varepsilon/5. \]
Observe that for $k>K$
\[ |\nu''(C') - \nu''_k(C')| \leq 
|\nu''(C') - \nu(C')| + |\nu(C') - \nu_k(C')| + |\nu_k(C') - \nu''_k(C')| \leq \]
\[ \leq \nu'(C') + \varepsilon/5 + |\nu_k(C\setminus A_n) - \nu_k(C\setminus (A_n \cup A_k))| \leq \]
\[ \leq 2\varepsilon/5 + \nu_k(A_k\setminus A_n) < 3\varepsilon /5. \]
Therefore, $|\nu''(C) - \nu''_k(C)| < \varepsilon$ for each $k>K$ and $(\nu''_k)_k$ converges to $\nu''$. Hence, the sequence $(\nu_k/\nu''_k(K))_k$ witnesses that $\nu''/\nu''(K)\in S_\alpha(K)$.
\end{proof}
\noindent
In what follows, according to the notation of Theorem \ref{mainfromone}, by $K$ and $\mu$ we will mean $K^{(\alpha)}$ and $\mu^{(\alpha)}$ for some $\alpha\in \om_1$. In the case when $\alpha = 1$ by $B_n$ we will mean the singleton $\{n\}$. We will
also denote $\bigcup_{j\leq n} B_n$ by  $B'_n$. Recall that by $R$ we denote the closed subspace of $K$ which supports $\mu$.\\\\
Every measure $\nu\in P(K)$ can be decomposed in the following way: $\nu = \nu^0+\nu^1+\nu^2$, where $\nu^0$ is the $L_0$ part of $\nu$, $\nu^1$ is the atomic part of $\nu-\nu^0$ and $\nu^2$ is the non-atomic part of $\nu-\nu^0$. Of course here $\nu-\nu^0$ is just the $L_1$ part of $\nu$. We shall say that $\nu$ is \emph{almost absolutely continuous with respect to $\rho$}, and we write $\nu\ll^* \rho$, if
$\nu^2 \ll \rho$.\\\\
The following result says that $\ll^*$, with respect to some $K$ and $\mu$, is inherited by limits of convergent sequences of measures.
\begin{thm} \label{first}
Let $(\nu_n)_n$ be a sequence of measures in $P(K)$. If for each $n$ we have $\nu_n\ll^* \mu$ and if $(\nu_n)_n$ converges to a non-atomic $\nu\in P(K)$, then $\nu \ll^* \mu$.
\end{thm}
\noindent
We postpone the proof of Theorem \ref{first} until the end. For now we have the following corollary.
\begin{cor}\label{everythingisalmostabsolutlycontinuoustomu}
	If $\nu\in S(K)$ then $\nu \ll^* \mu$.
\end{cor}
\begin{proof}
Of course, if $\nu\in S_0(K)$, then $\nu$ is purely atomic and thus $\nu\ll^*\mu$. Suppose
now that $\nu\in S_\alpha(K)$ and for each $\beta<\alpha$ all measures from $S_\beta(K)$ are almost absolutely continuous with respect to $\mu$.
By Lemma \ref{difference} the non-atomic part $\nu'$ of $\nu$ is in $S_\alpha(K)$.
Let $(\nu'_n)_n$ be a sequence from $\bigcup_{\beta<\alpha}S_\beta(K)$ convergent to $\nu'$.
According to Theorem \ref{first} we have $\nu'\ll^* \mu$. But this means that $\nu\ll^* \mu$.
\end{proof}
\noindent
With this we can complete the proof of Theorem \ref{mainfromone}.
\begin{proof}[Proof of (\ref{m2})]	
Let $\nu \in S(K^{(1)})$ and let $\nu'$ and $\nu''$ be the atomic and non-atomic parts of $\nu$, respectively. Since all measures in $L_0(K^{(1)})$ are purely atomic we must have $\nu'' = \nu^2$. By Corollary \ref{everythingisalmostabsolutlycontinuoustomu} then $\nu''\ll \mu^{(1)}$ and so by Theorem \ref{plebanek} we must have $\nu''\in S_2(K^{(1)})$. Since $\nu'\in S_1(K^{(1)})$ we must have $\nu = \nu_1+\nu_2 \in S_2(K^{(1)})$.\\\\
Assume now that $2\leq \alpha<\omega_1$ and that for each $1\leq \beta<\alpha$ we have $S(K^{(\beta)}) = S_{\beta+1}(K^{(\beta)})$. Let $\nu \in S(K^{(\alpha)})$. Again by Corollary \ref{everythingisalmostabsolutlycontinuoustomu} we have $\nu^2\ll \mu^{(\alpha)}$. By the inductive hypothesis, for each $n$, we have $\nu^0\restriction B_n \in S_{\alpha}(K^{(\alpha)})$. Since $\nu^0$ is the limit of the measures 
$$
\sum_{k\leq n}\nu^0\restriction B_k,
$$
appropriately rescaled, we must have $\nu^0\in S_{\alpha+1}(K^{(\alpha)})$. Clearly $\nu^1\in S_1(K^{(\alpha)})$ and by Theorem \ref{plebanek} we have $\nu^2\in S_{\alpha+1}(K^{(\alpha)})$. Thus $\nu\in S_{\alpha+1}(K^{(\alpha)})$ and we are done.
\end{proof}
\noindent
Towards the proof of Theorem \ref{first}, we will first consider the particular case when the measures concerned are from $L_0$.
\begin{prop} \label{L0}
Let $(\nu_n)_n$ be a sequence of measures in $P(K)$. If each $\nu_n\in L_0$ and $(\nu_n)_n$ converges to a measure $\nu$, 
then $\nu \ll^* \mu$.
\end{prop}
\begin{proof}
Assume for the contradiction that there is an $\varepsilon>0$ such that for each $i\in\scr{N}$ there is $Z_i\in \scr{B}$ such that $\mu(Z_i) <1/(i+1)$ and $\nu^2(Z_i)>\varepsilon$. Let $Z_0=\om$. If $\nu^0(K) = 0$ then, in what follows, let $r = 0$. Otherwise choose $r>0$ such that  $r<\nu^0(K)<r+\varepsilon/2$. See to it that $r+\varepsilon <1$.\\\\
We construct inductively increasing sequences $(n_i)_i$ and $(k_i)_i$ of natural numbers such that
\begin{itemize}
    \item[(1)] $\mu_n(Z_i)<1/(i+1)$ for each $n>n_i$ with  $i>0$;
	\item[(2)] $\nu_{k_i}\left( B'_{n_0} \cup ((B'_{n_{i+1}} \setminus B'_{n_i}) \cap Z_i)\right) > r+\varepsilon$ for each even $i$;
    \item[(3)] $\nu_{k_i}( B'_{n_{i+1}} \setminus B'_{n_i})  \geq 1-r-\varepsilon/2$ for each odd $i$.
\end{itemize}
Let $n_0$ be such that $\nu^0(B'_{n_0})\geq r$ and take $k_0$ such that $\nu_k(B'_{n_0})\geq r$ for each $k\geq k_0$.
Let $n_1$ be such that $\mu_n(Z_1)<1/2$ for each $n>n_1$ and $\nu_{k_0}(B'_{n_1} \cap Z_0)>r+\varepsilon$. 
This is possible since $\nu_{k_0}(Z_0) = 1$ and $\nu_{k_0}\in L_0$.\\\\
Assume now
 that we have constructed $k_i$ and $n_{i+1}$ for even $i$. Since $\nu(B'_{ n_{i+1}}) = \nu^0(B'_{n_{i+1}})$ there exists $k_{i+1} > k_i$ such that 
\[ \nu_{k_{i+1}} (B'_{n_{i+1}}) < r + \varepsilon/2. \]
Now find $n_{i+2}>n_{i+1}$ and for every $n>n_{i+2}$ we have 
\[ \nu_{k_{i+1}}(B'_{n_{i+2}} \setminus B'_{n_{i+1}})\geq 1- r - \varepsilon/2\] and 
\[ \mu_n(Z_{i+2})<1/(i+3). \]
Since $\nu(Z_{i+2}\setminus B'_{n_{i+2}})\geq \nu^2(Z_{i+2}\setminus B'_{n_{i+2}}) = \nu^2(Z_{i+2})>\varepsilon$ we can find $k_{i+2} > k_{i+1}$ such that
\[ \nu_{k_{i+2}} \left( B'_{n_0} \cup (Z_{i+2} \setminus B'_{n_{i+2}})\right) > r + \varepsilon. \]
Let $n_{i+3}>n_{i+2}$ be such that \[ \nu_{k_{i+2}}\left(B'_{n_0}\cup ( (B'_{n_{i+3}} \setminus B'_{n_{i+2}}) \cap Z_{i+2})\right)>r+\varepsilon\] and 
\[ \mu_n(Z_{i+3})<1/(i+4) \] for every $n>n_{i+3}$.\\\\
Finally, let 
\[ E = B'_{n_0} \cup \left(\bigcup_{i} (B'_{n_{2i+1}} \setminus B'_{n_{2i}}) \cap Z_{2i}\right). \]
Then $E\in \scr{B}$ because of (1) and the fact that $E\cap B_n\in \scr{B}_n$ for each $n$. By (2) we have that $\nu_{k_i}(E)>r+\varepsilon$ for each even $i$ and (3) implies that $\nu_{k_i}(E)<r+\varepsilon/2$ for each odd $i$. This shows that the sequence 
$(\nu_k)_k$ does not converge, a contradiction.
\end{proof}\noindent
We will need three more lemmas.
\begin{lem}\label{findseq} If $\nu^2 \not\ll \mu$ then there exists a sequence of clopen sets $D_0\supseteq D_1\ldots$ 
	such that $\mu(D_n)\rightarrow 0$ and $\nu(\widehat{\bigcap}_n D_n)>0$, where the intersection is taken in $K$.
\end{lem}

\noindent
\begin{proof}
Let $D'\sub R$ be a Borel set such that $\mu(D')=0$ and $\nu^2(D')>0$. Since $\mu\restriction R$ is clopen regular we can find 
a sequence $(D'_n)_n$ from $\scr{B}$ such that $\mu(D'_n) \rightarrow 0$ and $D'\sub D'_n$ for each $n$. Now let $D_n = \widehat{\bigcap}_{i\leq n}D'_i$.
\end{proof}

\noindent
\begin{lem}\label{p4} If $V\subseteq K\setminus \widehat{\bigcup}_i B_i$ is a countable set such that for each $v\in V$ there exists an $A\in v$ such that $\mu(A) = 0$, 
then there exists a clopen $A$ such that $V\subseteq A$ and $\mu(A) = 0$.
\end{lem}
\noindent
\begin{proof} Let $V= \{v_i\colon i\in\omega\}$ and for each $i$ let $A_i\in \scr{B}$ be such $\mu(A_i) = 0$ and $A_i\in v_i$. Inductively choose a sequence of integers $n_0<n_1<n_2<\cdots$ such that for each $n\geq n_i$ we have $\sum_{m\leq i}\mu_n(A_m)<\frac{1}{i+1}$. Now set
$$
A  = \bigcup_{i} (  \bigcup_{m\leq i}A_m\cap \bigcup_{k\in [n_i,n_{i+1})}B_k ) = \bigcup_{i}(A_i\setminus \bigcup_{k<n_i} B_k).
$$
If $n\geq n_i$ then 
$$
\mu_n(A) = \mu_n(A\cap B_n) = \mu_n(\bigcup_{m\leq i}A_m)\leq \sum_{m\leq i}\mu_n(A_m)< \frac{1}{i+1}.
$$
Plainly $A\in\scr{B}$ and $\mu(A) = 0$. Moreover, for each $i$ we have $A_i\setminus A = A_i\cap \bigcup_{k < n_i}B_{k}$. Since $v_i\not\in \widehat{\bigcup}_iB_i$ we must have $v_i\in A$.
\end{proof}
\begin{lem}\label{newlem} Let $V$ be a countable subset of $K$ disjoint from $\widehat{\bigcup}_{i} B_i\cup R$. Then there exists a closed extremely disconnected set $E\subseteq K$ such that $V\subseteq E$.
\end{lem}
\noindent
\begin{proof} 
	\noindent
For each infinite subset $D$ of $\omega$ fix the enumeration $\{x_0^D<x^D_1<\ldots\}$.  Let $\alpha<\omega_1$ and let $D\sub K^{(\alpha)}$ be a clopen homeomorphic to $K^{(\xi)}$ for some $ 1\leq \xi\leq \alpha$. For $n\in\om$ define  \[B_n(D) = \{x_i^D\colon i\in B_n\}. \]
If $D$ is homeomorphic to $K^{(1)}$, then $B_n(D) = \{x_n^D\}$ and if $k\in\om$, then let $B_n(\{k\}) = \{k\}$.
Denote $\mathcal{B}(D) = \{B_0(D), B_1(D), \dots\}$. Loosely speaking
if $D$ is a ``brick'' used in the recursive construction of $K$ at some level, then $\mathcal{B}(D)$ denote the bricks used \emph{directly} for the construction of $D$. Now, let $\mathcal{B}_0 = \{K\}$, and define inductively \[\mathcal{B}_{n+1} =
\bigcup\{ \mathcal{B}(D)\colon
D\in \mathcal{B}_n\}.\] So,
$\mathcal{B}_1= \{B_0, B_1, \dots\}$, $\mathcal{B}_2$ is the family of all ``bricks'' used directly for the construction of $B_0$, $B_1$, \dots \ and so on. Each $\mathcal{B}_n$ is a partition of $\om$ into clopen subsets of $K$. For each $D\in
\bigcup_{n\in\omega}\mathcal{B}_{n}$ there is $\xi(D)\leq\alpha$ such that $D$ is a homeomorphic copy of $K^{(\xi(D))}$. 
Also, each $D\in \bigcup_{n\in\omega}\mathcal{B}_{n}$ carries a canonical measure $\mu^D$, analogous to the measure $\mu$ for
$K$ (if $D$ is a singleton of a natural number $k$, then put $\mu^D=\delta_k$). For each $n\in\omega$ let 
$$
\mathcal{I}_n = \{I\in \mathrm{Clopen}(K)\colon \mu^D(I\cap D)=0 \mbox{ for each }D\in \mathcal{B}_n\}.
$$
\noindent
Each member of $V$ contains an element of $\mathcal{I}_0 = \mathrm{Null}(\mu)$. By Lemma \ref{p4} we can find a clopen $A_0\in \mathcal{I}_0$ such that $V\subseteq A_0$.
\begin{claim}\label{w4} If $A\in \mathcal{I}_n$ for each $n$ then $A$ is homeomorphic to $\beta\omega$.
\end{claim}
\noindent 
\begin{proof} Let $C\subseteq A$ and suppose that $C$ is not clopen. Since $A \in \mathcal{I}_0$ we know that for some $D_1\in \mathcal{B}_{1}$ we have $C\cap D_1$ is not clopen. If $\xi(D_1)= 1$ then every subset of $D_1\cap A$ would be clopen, so $\xi(D_1)>1$.
	Since $A \in\mathcal{I}_1$ we can find some $D_2\in \mathcal{B}_{2}$ such that $D_2\cap C$ is not clopen and $\xi(D_1)>\xi(D_2)$ (and $D_1\supseteq D_2$). 
As before we see that
$\xi(D_2)>1$. In this way we construct a sequence of non-zero ordinals $\xi(D_1)>\xi(D_2)>\xi(D_3)>\ldots$ which is a contradiction. 
\end{proof}
\bigskip

\noindent
If $A_0$ is homeomorphic to $\beta\omega$ then we may take $P = \emptyset$ and $O = A_0$ and we are done. So by Claim \ref{w4}, we may assume that we can find
$$
n_0 = \max\{n\colon A_0\in \mathcal{I}_n\}.
$$
If $G\subseteq K$ is clopen and $\mathcal{J}\subseteq \mathrm{Clopen}(K)$ is an ideal such that $G\not\in \mathcal{J}$, then by $G/\mathcal{J}$ we shall mean the Boolean algebra $\mathrm{Clopen}(G)/\mathcal{J}$.
\begin{claim}\label{w2} Let $A\in \mathcal{I}_{n}\setminus \mathcal{I}_{n+1}$, then $A/\mathcal{I}_{n+1}$ is isomorphic to the random algebra.
\end{claim}
\noindent
\begin{proof} For each $D\in \mathcal{B}_{n+1}$ the algebra $D/\{I\cap A\cap D\colon I\in \mathcal{I}_{n+1}\}$ is either isomorphic to the random algebra or trivial. 
Since $A\notin \mathcal{I}_{n+1}$ at least for one $D\in \mathcal{B}_{n+1}$ the above is isomorphic to the random algebra. Thus,
$A/\mathcal{I}_{n+1}$ is just the countable (perhaps, finite) product of random algebras, which is again isomorphic to the random algebra.
\end{proof}
\bigskip

\noindent
By Claim \ref{w2} we know that $A_0/ \mathcal{I}_{n_0+1}$ is the random algebra. Define $P_0 = \mathrm{Stone}(A_0/\mathcal{I}_{n_0+1})$.
\begin{claim}\label{w1} Let $A\in \mathrm{Null}(\mu)$ and let $n\in \om$. If $U\subseteq A$ is a countable set disjoint from  $\widehat{\bigcup}_i B_i$ such that each member of $U$ contains a member of $\mathcal{I}_n$, then there exists an $O\in \mathcal{I}_n$ such that $U\subseteq O$.
\end{claim}
\noindent 
\begin{proof} Enumerate $U = \{u_0,u_1,...\}$. For each $i$ let $O_i \in \mathcal{I}_n\cap u_i\cap \cP(A)$. Let \[O = \bigcup_{i} \left(  O_i\setminus \bigcup_{k\leq i} B_k \right).\] Since $A\in\mathrm{Null}(\mu)$ we know that $O$ is clopen, and
		clearly $u_i\in O$ (since no $u_i$ is a member of $\widehat{\bigcup}_i B_i$). Now for any $D\in \mathcal{B}_n$, since $D\subseteq B_i$ for some $i$, we have $O\cap D\subseteq O\cap B_i \subseteq \bigcup_{k<i} O_k$. But then 
$$
\mu^D(O\cap D)\leq  \sum_{k<i} \mu^D(O_k) = 0,
$$
and $O\in \mathcal{I}_n$.
\end{proof}
\bigskip

\noindent
By Claim \ref{w1}, we can find a clopen $A_1\subseteq A_0$ such that $V\setminus P_0\subseteq A_1$. If $A_1$ is homeomorphic to $\beta\omega$ then we may take $P = P_0$ and $O = A_1$, and we are done. So we may assume, by Claim \ref{w4}, that we can find 
$$
n_1 = \max\{n\colon A_1\in \mathcal{I}_n\}.
$$
Let $P_1 = \mathrm{Stone}(A_1/\mathcal{I}_{n_1+1})$ and notice that, as before, it is homeomorphic to the Stone space of the random algebra.\\\\
Continuing in this way we may, without loss of generality, assume that there exists a sequence of clopens $(A_i)_{i\in \omega}$ and a sequence of closed sets $(P_i)_{i\in\omega}$ such that $V\subseteq A_0$ and for each $i$ we have:
\begin{itemize}
\item $A_i\supseteq A_{i+1}$;
\item $A_i\in \mathcal{I}_i$;
\item $P_i$ is homeomorphic to the Stone space of the random algebra;
\item $P_i\subseteq A_i\setminus A_{i+1}$;
\item $(A_i\setminus P_i)\cap V\subseteq A_{i+1}$, in particular $(A_i\setminus (A_{i+1}\cup P_i))\cap V = \emptyset$.
\end{itemize}
\noindent
Let $P = \mathrm{Closure}(\widehat{\bigcup}_i P_i)$.
\begin{clm} $P$ is extremely disconnected.
\end{clm}
\noindent
\begin{proof}
\noindent 
Modifying the proof of Claim \ref{w4} it can be shown that if $C_i \sub A_i$ for each $i$, then $C = \bigcup_i C_i$ is a clopen subset of $K$. 
So, the clopen algebra of $\mathrm{Closure}(\widehat{\bigcup}_n P_n)$ is isomorphic to the product of $\prod_n A_n/\mathcal{I}_{n+1}$. Hence, it is isomorphic to the countable product of random algebras, and so it is isomorphic to the random algebra.
\end{proof}
\bigskip

\noindent
Now for each $v\in \widehat{\bigcap}_i A_i\setminus P$ we can find clopen set $O_v$, disjoint from $P$ and containing $v$. In particular each $O_v$ is disjoint from each $P_i$ and therefore $O_v\in \bigcap_{i} \mathcal{I}_i$. We can of course assume
that each $O_v$ is a subset of $A_0$. Just as in Claim \ref{w1}, we can find a clopen $O$ such that $(\widehat{\bigcap}_i A_i\setminus P)\cap V\subseteq O\subseteq A$ and $O\in \bigcap_{i} \mathcal{I}_i$. By Claim \ref{w4} we see that $O$ is homeomorphic to $\beta\omega$. Let $E = O\cup P$. Now if $v\in A_0\setminus E$ then for some $i$ we have that $v\in (A_i\setminus (A_{i+1}\cup P_i))\cap V = \emptyset$, and this is a contradiction. Thus $V\subseteq E$, which completes the proof of Lemma \ref{newlem}.
\end{proof}

\noindent
We are now ready to prove Theorem \ref{first}. 
\begin{proof}[Proof of Theorem \ref{first}] Let $A$ be the collection of ultrafilters used to define each of the $\nu^1_k$. There is a closed extremely disconnected space $R' (= E \cup R$, where $E$ is given by Lemma \ref{newlem})\ such that $A\subseteq R'$. Assume for a contradiction that $\nu^2$ is not absolutely continuous with respect to $\mu$. We construct two strictly increasing sequences of natural numbers $(k_i)_{i}$ and $(n_i)_{i}$, two sequences of clopen sets $(F_i)_{i}$ and $(Z_{i})_{i}$ and strictly positive reals $\alpha,\delta,\epsilon$ with 
\begin{equation}\label{p45}
\mbox{$\alpha+8\delta <2\epsilon$ and $\delta<\epsilon<\alpha$} 
\end{equation}
and such that the following holds for each $i$:
\begin{enumerate}[(a)]
	\item $Z_{i}\supseteq Z_{i+1}$;
        \item $F_i\subseteq Z_{i}\setminus Z_{{i+1}}$;
	\item $\mu_n(Z_{i})<\frac{1}{i+1}$, for each $n>n_i$;
	\item $\nu_k(Z_0) < \alpha$, for each $k$;	
	\item $\nu^0_{k_{i+1}}\left(Z_{i}\cap (B'_{n_{i+1}}\setminus B'_{n_i})\right) > \nu^0_{k_{i+1}}(Z_{i}) - \delta$;
        \item $\nu^1_{k_{i+1}}(R'\cap F_i) > \nu^1_{k_{i+1}}(Z_i) - \delta$;
	\item $\nu^2_{k}(R'\cap Z_{i}\setminus Z_{{i+1}}) > \nu^2_{k}(Z_{i}) - \delta$, for each $k\geq k_{i+1}$;
	\item $\nu_{k}(Z_{i}) > \epsilon - \delta$, for each $k\geq k_{i+1}$.	
\end{enumerate}
\noindent
Assuming the above for now, define for each $i$ the set
\[G_i = R' \cap (Z_{i}\setminus Z_{{i+1}}).\]
Let $G\in \scr{B}$ be such that $G\cap R' = \bigvee_i G_{2i}$. Here ``$\bigvee$'' denotes the supremum taken in the complete Boolean algebra $\mathrm{Clopen}(R')$. For each $i$ consider the set
\[E_i = Z_{i}\cap (B'_{n_{i+1}}\setminus B'_{n_i}). \]
Let
\[ H = \left(G\setminus \bigcup_i E_{2i+1}\right)\cup \bigcup_i E_{2i}.\]
Notice that $H\in \scr{B}$. 
Indeed, $\bigcup_i E_{2i}\cap B_n\in \scr{B}_n$ for each $i$ and $n$ and $\mu_n(E_{2i})$ converges to $0$ for every $i$, 
so $\bigcup_i E_{2i}\in \scr{B}$. The same argument works for $\bigcup_i E_{2i+1}$.\\\\
\noindent
Assume that $i$ is even and let $j=k_{i+1}$. Then 
\begin{itemize}
	\item $E_i \sub H$ and thus, by the definition of $E_i$ and by (e), $\nu^0_{j}(H)>\nu^0_{j}(Z_{i})-\delta$;
	\item $(F_{i}\cap R')\setminus \widehat{\bigcup}_i B_i \sub H$ so, by (f), $\nu^1_{j}(H)>\nu^1_{j}(Z_{i})-\delta$;
	\item $G_i \setminus \widehat{\bigcup}_i B_i \sub H$ so, by the definition of $G_i$ and by (g), $\nu^2_{j}(H)>\nu^2_{j}(Z_{i})- \delta$.
\end{itemize}
Therefore 
\[ \nu_j(H)>\nu^0_j(Z_{i})+\nu^1_j(Z_{i})+\nu^2_j(Z_{i}) - 3\delta = \nu_j(Z_{i})- 3\delta > \epsilon - 4\delta. \]
Similarly if $i$ is odd and $j = k_{i+1}$, then 
\[ \nu_j(Z_0\setminus H)> \epsilon - 4\delta \]
and, by (d), 
\[ \nu_j(H)<\nu_j(Z_0) - \nu_j(Z_0\setminus H) < \alpha - \epsilon + 4\delta. \]
Thus for every $i$ and $j$ we have
$$
\nu_{k_{2i+1}}(H) < \alpha -\epsilon + 4\delta <\epsilon - 4\delta <\nu_{k_{2j}}(H).
$$
This shows that the sequence $(\nu_k)_{k}$ does not converge and we have our contradiction.\\\\
Let us now construct the above sequences and reals. We may assume that $\nu^2(R') = 1$. Indeed, let $J\in \scr{B}$ be such that $\mu(J)=0$. Since $A\subseteq R'$ and for each $k$ we have $\nu^2_k\ll \mu$, we must have $\nu_k\restriction J \in L_0$. According to Proposition \ref{L0} (normalising if necessary) the limit of $\nu_k\restriction J$ is almost absolutely continuous with respect
to $\mu$, so it has to belong to $L_0$. Thus, $\nu^2(J)=0$. Since $\nu$ is non-atomic it follows that $\nu^2(R')=1$.
\\\\
By Lemma \ref{findseq} we can find a sequence $D_0\supseteq D_1\supseteq \ldots $ from $\scr{B}$ such that $\mu(D_i)<\frac{1}{i+1}$ and $\nu^2(\widehat{\bigcap}_{i}D_i)>0$. Let $A_i = D_i\setminus B'_i$. Since for each $i$ we have $A_i\subseteq D_i$ and $\nu^2(A_i) = \nu^2(D_i)$ we still have $\mu(A_i)<\frac{1}{i+1}$ and $\nu^2(\widehat{\bigcap}_{i}A_i) = \nu^2(\widehat{\bigcap}_{i}D_i)> 0$. Let $\epsilon = \nu(\widehat{\bigcap}_{i}A_i)$. Now choose any $\delta$ and $\alpha$ satisfying (\ref{p45}).\\\\
Let $N$ be such that $\nu^0(B'_N)>\nu^0(K)-\delta$. By truncating the sequence $(A_i)_{i}$ we can assume that $A_0\cap B'_N = 0$ and $\nu(A_0)<\alpha$. By truncating the sequence $(\nu_k)_{k}$ we may assume that $\nu_k(A_0)<\alpha$ for each $k$. Since $\nu^2_k\ll \mu$ we can further assume (by considering a subsequence of $(A_i)_{i}$ if necessary) that $\nu^2_k(A_k)<\delta$ for each $k$.\\\\
Let $(\delta_i)_{i}$ be a sequence of strictly positive reals such 
that $\sum_{i} \delta_i=  \delta$. Let $k_0 = 0$ and $n_0>N$ (so that $\nu^0(B'_{n_0}\setminus B_{N}')<\delta$). 
Let $Z_0 = A_0$. We proceed by induction. Suppose that we have constructed $(F_l)_{l < i}$, $(Z_l)_{l\leq i}$, $(k_l)_{l\leq i}$, $(n_l)_{l\leq i}$ satisfying conditions (a)--(h) and such that
\begin{itemize}
 \item $\nu(F_j) <\delta_j$, for each $j<i$;
 \item $Z_j = A_{k_j}\setminus \bigcup_{l<j} F_l$, for each $j\leq i$.
\end{itemize}
Observe that \[ \nu(Z_i) = \nu({A}_{k_i}\setminus \bigcup_{j< i}F_j)>{\epsilon} -\sum_{j< i}\delta_j > {\epsilon}
- \delta \] so we can find $M$ such that 
\begin{equation}\label{ph7}
\nu_k(Z_{i})>{\epsilon}-\delta
\end{equation}
for each $k\geq M$.\\\\
Since for each $k$ we have $\nu^2_k(A_k) < \delta$ it follows that 
\begin{equation}\label{ph8}
\nu^2_k(Z_i\setminus A_k) > \nu^2_k(Z_i)-\delta.
\end{equation}
Since $n_i>N$ we can find $k_{i+1}>\max\{M,k_i\}$ such that for all $k\geq k_{i+1}$ we have
\begin{equation}\label{pp3}
\nu^0_k(Z_i\setminus B'_{n_i}) > \nu^0_k(Z_i)-\delta.
\end{equation}
Let $F_{i}\sub Z_i$ be a clopen such that 
\begin{equation}\label{ph4}
\nu^1_{k_{i+1}}(F_{i}\cap R')\geq \nu^1_{k_{i+1}}(Z_i)-\delta
\end{equation}
and $\nu(F_{i})<\delta_{i}$. We can do this since $\nu^1_{k_{i+1}}$ is purely atomic and $\nu$ is non-atomic.\\\\
Now set $Z_{i+1} = A_{k_{i+1}} \setminus \bigcup_{j\leq i} F_j$. By (\ref{pp3}) and the fact that $\nu^0_{k_{i+1}}\in L_0$, there is $n_{i+1}>n_i$ such that 
\begin{equation}\label{ph5}
 \nu^0_{k_{i+1}}\left(Z_{i} \cap (B'_{n_{i+1}}\setminus B'_{n_i})\right)>\nu^0_{k_{i+1}}(Z_{i})-\delta
\end{equation}
and
\begin{equation}\label{ph6}
\mu_n(Z_{{i+1}})<\frac{1}{i+2} 
\end{equation}
for each $n>n_{i+1}$. Now proceed by induction.\\\\
Finally, conditions (a), (b) and (d) are clear, (\ref{ph6}) implies (c), (\ref{ph5}) implies (e), (\ref{ph4}) implies (f), (\ref{ph7}) implies (h), and (\ref{ph8}) together with $Z_{i+1}\sub
A_{k_{i+1}}$ and the fact that $\nu^2_{k_{i+1}}(R')=1$ implies (g).
\end{proof}

\noindent
We finish this section with two open problems. The first one is obvious. 
\begin{prob}
Can we improve Theorem \ref{mainfromone} to show that for each $1\leq \alpha <\omega_1$ we have $S(K^{(\alpha)}) = S_\alpha(K^{(\alpha)})$?
\end{prob}
\noindent
With regards to Theorem \ref{grzes} we have the following.
\begin{prob}
Is it provable in $\ZFC$ that for each $\alpha\leq \om_1$ there is a compact $K$ such that $S_\alpha(K) = P(K)$ with
$$
S_\alpha(K)\setminus \bigcup_{\beta<\alpha}S_\beta(K) \neq \emptyset.
$$
\end{prob}
\noindent
Of course, one can replace the space of probability measures by any topological space and consider the corresponding sequential hierarchy there. Indeed, given a topological space $K$ and a set $A\subseteq K$ we let $\mathsf{Seq}_0(A) = A$. If $\alpha\leq \omega_1$ and for each $\beta<\alpha$ we have defined $\mathsf{Seq}_\beta(A)$ then we let 
$$
\mathsf{Seq}_\beta(A) = \{\lim_n x_n\colon (\forall n)(\exists \beta<\alpha)(x_n\in \mathsf{Seq}_\beta(A))\}.
$$
It is worth recalling here that in \cite[Section 12]{OpenProblems} it is asked if in $\ZFC$ there is a sequentially compact space such that
$$
\sup_{A\subseteq K}\min\{\alpha:\mathsf{Seq}_\alpha(A) = \mathsf{Seq}_{\omega_1}(A)\}\geq 2.
$$
\section{Acknowledgements}
\noindent
Some of the results in this article were obtained during a visit to the Mathematical Institute at the University of Wroc\l aw by the second author. The second author wishes to thank the Mathematical Institute at Wroc\l aw for their hospitality and acknowledge the support of the INFTY Research Networking Program (European Science Foundation) via grant 3548. The authors would like to thank Grzegorz Plebanek for his stimulating discussion and valuable remarks. The authors would also like to thank the anonymous referee for pointing out a significant gap in their arguments in a previous version of this article, and for his helpful comments.

\nocite{Fremlin-Cp}
\bibliographystyle{alpha}
\bibliography{fbib}

\end{document}